\documentclass[11pt]{article}
\usepackage[utf8]{inputenc}
\usepackage[T1]{fontenc}
\usepackage{textcomp}
\usepackage[english]{babel}
\usepackage[autolanguage]{numprint}
\usepackage{hyperref}
\usepackage{graphicx}
\usepackage{verbatim}
\usepackage{xcolor}

\usepackage{amsmath,amssymb,amsthm}

\renewcommand\epsilon\varepsilon
\renewcommand\phi\varphi
\renewcommand\geq\geqslant
\renewcommand\leq\leqslant
\renewcommand\ln\log

\newcommand{\interior}[1]{%
	{\kern0pt#1}^{\mathrm{o}}%
}

\newcommand{\E}{\mathbb{E}}

\newcommand{\tbar}{\overline{\mathcal{T}}}

\newcommand\RR{\mathbb{R}}

\newcommand\ab\allowbreak

\newcommand{\cP}{\mathcal{P}}
\newtheorem{The}[equation]{Theorem}
\newtheorem{Lem}[equation]{Lemma}
\newtheorem{Pro}[equation]{Proposition}

\newtheorem{Def}[equation]{Definition}
\newtheorem{Rem}[equation]{Remark}

\numberwithin{equation}{section}

\newtheorem{Exe}[equation]{Example}

\title{From Hopf-Lax formula to optimal weak transfer plan}
\author{Yan Shu\footnote{Mod\'elisation al\'eatoire de Paris Ouest Nanterre La D\'efense(MODAL'X), Email: yan.shu.prof@gmail.com}}
\date{\today}
\begin{document}
\maketitle
\begin{abstract}
We study the properties of Hopf-Lax formula restricted to convex functions and provide a characterization of the optimal transfer plan for weak transport problems on the real line. On $n$ dimensional real space, we also provide a sufficient condition on the potential function such that the optimal plan of the classical Monge-Kantorovich problem does not depend on the cost function. As a byproduct, we establish a link between the combinatorial object (the permutation polytope) and the Hamilton Jacobi equation. 
\end{abstract}
key words:  Weak transport costs; Hopf-Lax formula; Monge-Kantorovich problem; Hamilton Jacobi equation

\section{Introduction}
\subsection{Monge-Kantorovich problem and weak transport cost}
\indent Throughout the paper, the space would be Euclidean space $\RR$ or $\RR^n$.  We denote $\cP(\RR^n)$ the set of probability measures on $\RR^n$,  $\cP_1(\RR^n)$ the set of probability measures on $\RR^n$ with finite first moment and $\|.\|$, the Euclidean norm.

Given two probability measures $\mu,\nu \in \cP_1(\RR^n)$, recall that the classical Monge-Kantorovich problem is to minimize the following transportation cost:
$$\mathcal{T}_c(\mu,\nu)=\inf_\pi\int c(x,y)d\pi\in [0,\infty],$$
where $c:\RR^n\times \RR^n \rightarrow \RR^+$ is called the cost function and $\pi$ is a measure coupling with marginal $\mu$ and $\nu$. Readers can refer to \cite{Villani} for an introduction of the optimal transport theory. One of the most interesting cases is that the cost function $c$ is a convex function with respect to the distance $d$.  One of the well-known results about the optimal transfer plan is due to Brenier \cite{B91} in 1991, he proved that on a finite dimensional real space, the optimal transfer plan is a gradient of some convex function. In particular, in one dimension, the optimal transfer plan does not depend on the cost function if the cost function is convex.  We also refer readers to related results about the optimal transfer plan in \cite{GM2000}, \cite{M95}, \cite{R97}.

 Here we are interested in the optimal transfer plan of a variant of Monge-Kantorovich problem, which is introduced in \cite{GRST15},  the general optimal transportation problem, defined by
 \begin{equation}\label{tbar}
\overline{ \mathcal{T}}_c(\nu|\mu)=\inf_{\pi}\int c(x,p)\,\mu(dx).
\end{equation}
The cost function $c$ is defined on the space $\RR^n\times \cP(\RR^n)$, the infimum runs over all couplings $\pi(dxdy)=p(x,dy)\mu(dx)$ of $\mu$ and $\nu$ and where $p(x,\,\cdot\,)$ denotes the disintegration kernel of $\pi$ with respect to its first marginal. Further papers directly related to this problem include \cite{BBP2018, BBP2019, FS15, GJ2018, GRSST15, samson2017, Sh15, SS16}.

In terms of random variables, one has the following interpretation:
\[
\overline{ \mathcal{T}}_c(\nu|\mu)=\inf \E \left(c(X,\E(Y|X))\right),
\]
whereas
\[
\mathcal{T}_c(\nu,\mu)=\inf \E \left(c(X,Y)\right),
\]
where in both cases the infimum runs over all random vector $(X,Y)$ such that $X$ has law $\mu$ and $Y$ has law $\nu$.  This general transport cost \eqref{tbar} plays an important role in the study of Talagrand type transport inequalities and some log-Sobolev inequalities, especially in a discrete space such as graphs or a subset of vector space. Those transport inequalities immediately yield concentration results and tensorization properties of the measure (see \cite{FS15,GRST15,GRSST15,Sh15,SS16}).

Throughout the  paper, the cost function is
of form:
$c(x,y)=\theta(x-y),$
where $\theta:\RR^n\mapsto \RR^+$ is a non-constant convex function vanishing at $0$ and radially symmetric with respect to the origin. In one dimension, $\theta$ is even, and in $n$ dimension, $\theta$ only depends on the distance between $x$ and  $y$.
 In what follows, taking two probability measures $\mu, \nu \in \cP_1(\RR^n)$, the optimal weak transport cost of $\nu$ with respect to $\mu$ and cost function $\theta:\RR^n\mapsto \RR$ means
\[
\overline{ \mathcal{T}}_\theta(\nu|\mu)=\inf_{\pi}\int \theta\left(x-\int y\,p(x,dy)\right)\,\mu(dx)
\]
where the infimum runs over all couplings $\pi(dxdy)=p(x,dy)\mu(dx)$ of $\mu$ and $\nu$, and where $p(x,\,\cdot\,)$ denotes the disintegration kernel of $\pi$ with respect to its first marginal.
Since the $\theta$ is convex,  by Jensen's inequality, one has
\[
\overline{ \mathcal{T}}_\theta(\nu|\mu)\leq \mathcal{T}_\theta(\nu,\mu).
\]

This weak transport cost is deeply linked to the Monge-Kantorivich problem. Together with Gozlan, Roberto, Samson and Tetali, following Strassen's theorem \cite{Str65}, we proved  in \cite{GRSST15} that,
\[
\tbar_\theta (\nu|\mu)=\inf_{\nu_1\preceq \nu} \mathcal{T}_\theta (\nu_1,\mu),
\]
where $\preceq$ is the convex order, defined as:
$$\nu_1\preceq \nu \qquad \Leftrightarrow \qquad \forall f\ convex, \int f d\nu_1\leq \int f d\nu.$$

  Moreover, in one dimension, the measure $\nu_1$ which achieves the infimum does not depend on the choice of the  convex cost function $\theta$. As an application, one can deduce a completed characterization of a convex modified Log-Sobolev inequality~\cite{SS16}.

\subsection{Presentation of the results}
It is natural to ask the following questions about the optimal transfer plan of weak transport problem:
\begin{itemize}
\item When the weak transport would be equal to the classical transport?
\item Does the optimal coupling depend on the cost function $\theta$?
\end{itemize}

We approach those questions by the following duality framework of the weak transport cost. Following \cite[Theorem 2.11]{GRST15}, for $\mu,\nu \in \cP_1(\RR^n)$, and $\theta:\RR^n\mapsto \RR$ a convex function such that $\theta(x)\geq a\|x\| + b$ for some $a>0$ and $b\in \RR$, it holds
\begin{equation}\label{duality}
\tbar_\theta(\nu|\mu)=\sup_f \int Q_1^\theta f d\mu-\int fd\nu 
\end{equation}
where the supremum runs over all function $f$ convex, Lipschitz and bounded from below, and where  
$$Q_t^\theta f(x)=\inf_{y\in \RR^n}\left\{f(y)+t\theta\left(\frac{y-x}{t}\right) \right\}.$$
The above equation is often called the Hopf-Lax formula, in some references the operator $Q_t^\theta$ is also called the inf-convolution operator and often denoted simply by $Q^\theta$ for $t=1$. The Hopf-Lax formula is known as the solution of a Hamilton Jacobi equation, and has been widely studied in many different contexts. We remark as well in the proof of \cite[Theorem 5.5]{BBP2018}, Backhoff-Veraguas, Beiglbock and Pammer shown that a maximizer of equation \eqref{duality} exists even with a weaker assumption. 

 In this paper, since the cost function $\theta$ is convex, real valued, positive and only depended on the euclidean norm, it satisfies the growth condition $\theta(x)\geq a\|x\| + b$ and the condition (A+) in \cite{BBP2018}. Therefore we can apply the duality theorem and the existence of the maximizer here.

Apart from the introduction, this paper is divided into three sections. In section 2, the space will be the real line $\RR$, and in section 3, the space will be $\RR^n$ and we will present some applications in section 4.

In section 2, we stay in one dimension. We provide an equivalent condition for the equality between the weak transport cost and the classical Monge-Kantorovich transport cost (theorem \ref{The:main1}), which states that
$$\tbar_\theta(\nu|\mu)=\mathcal{T}_\theta(\nu,\mu)$$
 if and only if the difference between the inverse cumulative functions of $\mu$ and $\nu$ is non-decreasing, precisely, the function $$F_\mu^{-1}-F_\nu^{-1}$$ is non-decreasing. This equivalence also has been obtained recently by Gozlan and Juillet in a slightly different form in \cite{GJ2018} and by Backhoff Veraguas,  Beiglboeck and Pammer in \cite{BBP2019}.

%
%

Furthermore, according to Brenier \cite{B91}, in one dimension, the optimal mapping of a classical transport does not depend on the cost function as soon as the cost function is convex.  In \cite{GRSST15}, the same result is obtained for weak transport cost. We will give a new proof of this result in section 2, using a very different argument. During this approach, a byproduct about the Hamilton-Jacobi equation is obtained, which might have its own interest (theorem \ref{cor:Q+}).

\medskip

In section 3, the space will be $\RR^n$.  We will show in theorem \ref{Themain1dim_n} an extension of the results of section 2, which provide equivalence conditions (under some smoothness properties) such that $\tbar_\theta(\nu|\mu)=\mathcal{T}_\theta(\mu,\nu)$ holds. We will define Property $\mathcal{F}$, which reduces to convexity in one dimension and plays the role of the convexity in $n$ dimensions. The condition that there is equality between the weak transport cost and the classical transport cost
 is deeply related to Property $\mathcal{F}$.

In section 4, we briefly explain some applications on the infimum convolution inequality introduced by Maurey in \cite{Mau91}.

\section{Weak transport on the line}
In this section we will focus on the real line.
A cost function in this section is a function $\theta:\RR\mapsto \RR^+$, convex, even, satisfying $\theta(0)=0$. 

We remind here the notion of cyclical monotone of a multivalued mapping~\cite[Page 238]{R70}. A multivalued mapping $\rho$ from $\RR^n$
to $\RR^n$ is called cyclically monotone if one has
$$\langle x_1 - x_0, x_0^*\rangle + \langle x_2 - x_1, x_1^*\rangle +\cdots + \langle x_0 - x_m, x_m^*\rangle \leq 0.$$
for any set of pairs $(x_i, x_i^*)$, $i = 0, 1, ... , m$ ($m$ arbitrary) such that
$x_i^* \in \rho(x_i)$. 
The property of cyclical monotone characterizes the gradient of a convex function according to \cite[Theorem 5.6]{R70}, which states that, $\rho$ is cyclically monotone if and only if there exist a convex proper function $f$ on $\RR^n$ such that $\rho \subset \partial f$.  
\subsection{A remark on Hopf-Lax formula}
We begin with some development of the Hopf-Lax formula. The key observation is the following lemma.

\begin{Lem}\label{The:main}
Let $I,J$ be two subset of $\RR$ with strictly positive Lebesgue measure. Given $t>0$, let $T_t:I\subset \RR \rightarrow J\subset \RR$ be a real valued function.  Then the following statements hold.

$(i)$ If there exists a real-valued convex function $f$ defined on $\RR$ and a strictly convex cost $\theta$ (recall that a cost function here should be positive and vanishing at $0$) such that for all $x\in I$, it holds
\begin{equation}\label{star}
Q_t^\theta f(x)= f(T_t(x))+t\theta\left(\frac{T_t(x)-x}{t}\right).
\end{equation}
then $T_t$ is non-decreasing and $x\mapsto T_t(x)-x$ is non-increasing.

$(ii)$ Inversely, taking a function $T_t$ defined on $I\subset\RR$, if $T_t$ is non-decreasing and $x\in J \mapsto T_t(x)-x$ is non-increasing, then for all convex cost $\theta$, there exists a closed proper convex function $f$, such that \eqref{star} holds for all $x\in I$.
\end{Lem}
\begin{proof}
$(i)$. Given $t>0$ and $x\in I$,  the function $G_x:u\mapsto f(u)+t\theta\left(\frac{u-x}{t}\right)$ is defined on $\RR$. According to convexity of $f$ and $\theta$, $G$ is strictly convex. Therefore, it has at most one minimum. Since a convex function always has a right derivative and a left derivative, then it holds for all $x\in I$:
\begin{equation}\label{carTt}
0\in [\partial_- G_x(T_t(x)),\partial_+ G_x(T_t(x))].
\end{equation}
We deduce that for all $x\in I$:
$$\partial_+ f(T_t(x))+\partial_+\theta\left(\frac{T_t(x)-x}{t}\right)\geq 0.$$
We will first prove that $T_t$ is non-decreasing.
For any $x,y\in I$, $x<y$, since $\theta$ is strictly convex, we deduce that
\begin{align*}
&\partial_+G_x(T_t(y))=\partial_+ f(T_t(y))+\partial_+\theta\left(\frac{T_t(y)-x}{t}\right)\\
         &> \partial_+f(T_t(y))+\partial_+\theta\left(\frac{T_t(y)-y}{t}\right)\geq 0 \geq \partial_+G_x(T_t(x)).
\end{align*}
Together with the strict convexity of $G_x$, we deduce that $T_t(y)\geq T_t(x)$. 

Now we turn to prove that $x\mapsto T_t(x)-x$ is non-increasing on $I$. We will show that for $x,y\in I$ such that $x<y$, it holds $T_t(x)-x\geq T_t(y)-y$. Since $T_t$ is non-decreasing, $T_t(x)\leq T_t(y)$.

If $T_t(x) = T_t(y)$, it follows immediately $T_t(x)-x>T_t(y)-y$. 

Now we assume that $T_t(x)<T_t(y)$. Since $f$ is convex, it holds 
\begin{equation}\label{eqff}
\partial_+f(T_t(x))\leq \partial_-f(T_t(y)).
\end{equation}
Now applying \eqref{carTt} for $x$ and $y$, we deduce:
\begin{equation}\label{eqgg}
   \partial_+ G_x(T_t(x))\geq 0 \geq \partial_- G_y(T_t(y)).
\end{equation}
Combining equation \eqref{eqff} and \eqref{eqgg},  we deduce
$$
\partial_+\theta\left(\frac{T_t(x)-x}{t}\right)\geq 
\partial_-\theta\left(\frac{T_t(y)-y}{t}\right).
$$
By convexity of $\theta$, it holds $T_t(x)-x\geq T_t(y)-y$. Thus $x \mapsto T_t(x)-x$ is non-increasing on $I$.

 $(ii)$. Assume that the function $T_t$ is non-decreasing on $I$ and the function $x\mapsto T_t(x)-x$ is non-increasing on $I$. 
Define a multi valued mapping $\rho$ from $\RR$ to $\RR$ such that
$$y^*\in \rho(y) \Leftrightarrow y^* =  -\partial_+\theta\left(\frac{y-x}{t}\right),$$
for some $x$ such that $T_t(x) = y$.

For any set of pairs $(y_i,y_i^*)$, $i = 0,1,...,m$, ($m$ arbitrary) such that $y_i^*\in \rho(y)$, there exist $x_0,...,x_m$, such that $T_t(x_i) = y_i$. We claim that $y_i< y_j$ implies $y_i^*\leq y_j^*$, for any $i,j \in \{0,1,...,m\}$. In fact, if $y_i<y_j$, then by monotony of $T_t$, it holds $x_i<x_j$. Since $x\mapsto T_t(x) - x$ is non-increasing, we deduce that $T_t(x_i)-x_i\geq T_t(x_j)-x_j$. Therefore, by convexity of $\theta$,  
$$y_i^* = -\partial_+\theta\left(\frac{T_t(x_i)-x_i}{t}\right)\leq -\partial_+\theta\left(\frac{T_t(x_j)-x_j}{t}\right) = y_j^*.$$
Therefore, according to rearrangement inequality, it holds:
\begin{align*}
     &\langle y_1 - y_0, y_0^* \rangle + \langle y_2 - y_1, y_1^*\rangle + \cdots + \langle y_0 - y_m,y_m^* \rangle \\
= & \sum_{i = 0}^m y_i y_i^* - \sum_{i = 0}^m y_i y_{i+1} \leq 0,
\end{align*}
where we denote $y_{m+1} = y_0$ and in dimension one, the bracket $\langle,\rangle$ is the multiplication. Thus, $\rho$ is cyclically monotone. According to \cite[Page 238, Theorem 24.8]{R70}, there exists a closed proper convex function $f$ (thus lower semi-continuous) defined on $\RR$, such that 
 for all $y\in \RR$, $\rho(y)\subset \partial f(y)$ ($\rho(y)$ could be empty). Therefore, for all $x\in I$, it holds
$$-\partial_+\theta\left(\frac{T_t(x)-x}{t}\right)\in [\partial_-f(T_t(x)),\partial_+f(T_t(x))],$$
which leads to
$$0\in \left[\partial_-f(T_t(x))+\partial_-\theta\left(\frac{T_t(x)-x}{t}\right),\partial_+f(T_t(x))+\partial_+\theta\left(\frac{T_t(x)-x}{t}\right)\right].$$
Thus equation \eqref{star} holds. The conclusion follows.
\end{proof}

\subsection{Characterization of equality between weak transport cost and transport cost}
We firstly recall the definition of cumulative distribution function and its inverse.

For a probability measure $\mu\in\mathcal{P}(\RR)$, denote $F_\mu$ the cumulative distribution function of $\mu$, i.e.
$$F_\mu(x)=\mu(-\infty,x],$$
and define the generalized inverse of $F_\mu^{-1}$ by
$$F_\mu^{-1}(t)=\inf \{x\in \RR; F_\mu(x)>t\}.$$

\begin{The}\label{The:main1}
Let $\mu,\nu\in\mathcal{P}_1(\RR)$ and let the cost function $\theta:\RR\mapsto \RR^+$ be even,  strictly convex and vanishing at $0$. Assume that $\mu$ and $\nu$ are absolutely continuous with respect to Lebesgue measure. Assume that the weak transport $\tbar_\theta(\nu|\mu)$ is finite, then
$$\tbar_\theta(\nu|\mu)=\mathcal{T}_\theta(\mu,\nu)$$
if and only if
$F^{-1}_\mu-F^{-1}_\nu$
is non-decreasing.
\end{The}


In one dimension, given two probability measures $\mu$ and $\nu$ which are absolutely continuous with respect to Lebesgue measure, the optimal transport mapping for transport problem is transporting mass from location $F_\mu^{-1}(x)$ to location $F_\nu^{-1}(x)$. This mapping is in fact the monotone rearrangement $T$ from $\mu$ to $\nu$. We will play with the Kantorovich duality formula related to this mapping $T$ for the weak transport problem.

\begin{proof}
We assume at first that $F_\mu^{-1}-F_\nu^{-1}$ is non-decreasing. Denote $T$ the monotone rearrangement from $\mu$ to $\nu$ and denote $I:=\mathrm{supp}(\mu)$ the support of $\mu$ and $J:=\mathrm{supp}(\nu)$ the support of $\nu$. Thus $T$ is an one-to-one mapping from $I$ to $J$, and for all $x\in I$, 
\begin{equation}\label{eqT}
F_\mu(x)=\mu(-\infty, x)=\nu(-\infty, T(x))=F_\nu(T(x)).
\end{equation}
Since $\mu, \nu$ are absolutely continuous, $F_\mu^{-1}$ and  $F_\mu^{-1}$ are well defined on $(0,1)$. 
Now given $u\in (0,1)$, denote $u = F_\mu(x) = F_\nu(T(x))$. According to the equality \eqref{eqT}, it holds
$$F_\nu^{-1}(u)-F_\mu^{-1}(u)=T(x) - x .$$
According to our hypothesis, $x\mapsto T(x)-x$ is non-increasing. Moreover, notice that $T$ is the monotone rearrangement from $\mu$ to $\nu$, $T$ is non-decreasing. We can apply lemma \ref{The:main}, and there exists a closed proper convex function $f$ (thus lower semi-continuous), such that for all $x\in I$:
$$Q^\theta f(x)=\inf_{y\in \RR} \{f(y)+\theta(y-x)\}=f(T(x))+\theta(T(x)-x).$$
Since $f$ is proper, there exist $x_0\in\RR$ such that $f(x_0)<\infty$. By definition of $Qf$, it holds
$$
Q^\theta f(x) = f(T(x))+\theta(T(x)-x)\leq f(x_0) + \theta(x_0 -x).
$$
Thus, for $R>|x_0|$ and for all $x\in I\cap [-R,R]$, 
$$f(T(x))\leq f(x_0) + \theta(x_0 -x)- \theta(T(x) -x) \leq f(x_0) + \theta(2R).$$
Therefore $f$ is a lower semi-continuous, bounded and lipschitz on the interior of $I\cap [-R,R]$. 

By Kantorovich duality from \cite{GRST15}, we have
\begin{align*}
\tbar_\theta(\nu|\mu)&=\sup_{\phi\ convex, Lip,\ l.s.c.} \int Q^\theta \phi d\mu-\int \phi d\nu \\
&\geq \int_{I\cap [-R,R]} Q^\theta f d\mu-\int_{I\cap [-R,R]} f d\nu=\int_{I\cap [-R,R]} \theta(T(x)-x)d\mu\\
&=\mathcal{T}_\theta(\mu,\nu)-\epsilon_R\geq \tbar_\theta(\nu|\mu)-\epsilon_R.
\end{align*}
Now let $R\rightarrow \infty$, $\epsilon_R\rightarrow 0$, and we obtain
$$\tbar_\theta(\nu|\mu)=\mathcal{T}_\theta(\mu,\nu).$$
Now assume that $\tbar_\theta(\nu|\mu)=\mathcal{T}_\theta(\mu,\nu)$.  
 According to \cite[Theorem 5.5]{BBP2018}, for $\mu,\nu\in \cP_1(\RR)$, if $\tbar_\theta(\nu|\mu)$ is finite, there exists a convex function $f$ such that 
$$\tbar_\theta(\nu|\mu)= \int Q^\theta fd\mu-\int fd\nu.$$
Thus,
\begin{align*}
\tbar_\theta(\nu|\mu)&= \int Q^\theta fd\mu-\int fd\nu \\
&\leq \int f(T(x))+\theta(T(x)-x)d\mu-\int fd\nu =\int \theta(T(x)-x)d\mu\\
&=\mathcal{T}_\theta(\mu,\nu).
\end{align*}
Therefore, $Q^\theta f(x)= f(T(x))+\theta(T(x)-x)$ $\mu$ almost-surely.
Since $f$ is convex and $\theta$ is strictly convex, according to lemma \ref{The:main} (i), we obtain that $T(x)-x$ is non-increasing $\mu$ almost surely, together with the monotonicity of the function $F_{\mu}^{-1}$,  the conclusion follows.
\end{proof}

\subsection{Weak transfer plan}
In this section we give an alternative proof of a theorem in \cite{GRSST15}, as a consequence of lemma \ref{The:main}, which is the following:

\begin{The}\label{add} \cite[Theorem 1.4]{GRSST15}
Let $\alpha, \beta$ and $\theta$ be convex cost functions  satisfying $\alpha+\beta=\theta$. Let $\mu,\nu \in \mathcal{P}_1(\RR)$.  Assume that $\tbar_\theta(\nu|\mu)$ is finite, it holds
\begin{equation}\label{eqbut}
\tbar_\theta(\nu|\mu)=\tbar_\alpha(\nu|\mu)+\tbar_\beta(\nu|\mu).
\end{equation}
\end{The}

We need to prove the following proposition at first.

\begin{Pro}\label{thm:Q+}
Let $\alpha, \beta, \theta:\RR \rightarrow \RR$ be convex cost functions, even and of class $C^1$ satisfying $\alpha+\beta=\theta$. We assume that $\theta$ is strictly convex.
Then for all convex function $f:\RR\rightarrow \RR$ bounded from below of class $C^1$, there exists convex functions $f_1, f_2:\RR\mapsto \RR$, bounded from below and of class $C^1$, such that
\begin{equation*}
f=f_1+f_2
\end{equation*}
and for all $t>0$
\begin{equation*}
Q_t^\theta f=Q_t^\alpha f_1+Q_t^\beta f_2
\end{equation*}
more precisely, $$f_1(x)=a+\int_0^x \alpha' \circ \theta'^{-1}\circ f'(u)  du, $$
with  any constant $a\in \RR$.
\end{Pro}

From this proposition, combining with the fact that $Q_t^\theta f$ is in fact the solution of the Hamilton Jacobi equation \eqref{HJe} \cite{evans10}, one deduces immediately of the following theorem.

\begin{The}\label{cor:Q+}
Let $\theta:\RR \rightarrow \RR$ be a strictly convex cost function with super-linear growth of class $C^1$ (i.e. $\theta'(x)$ goes to $\infty$ as $x$ goes to $\infty$) and $\theta^*$ be the Legendre transform of $\theta$. Consider the following Hamilton Jacobi equation:
\begin{equation}\label{HJe}
\begin{cases}
\partial_t v(x,t)+\theta^*(\partial_x v(x,t))=0 & (x,t) \in
 \RR \times(0,\infty)\\v(x,0)=f(x) & x \in \RR.
\end{cases}
\end{equation}
Assume that the initial function $f:\RR\rightarrow \RR$  is convex bounded from below of class $C^1$, then for all convex cost function $\theta_1, \theta_2$, with super-linear growth and satisfying $\theta_1+\theta_2=\theta$, there exists $v_1$, $v_2$ and $f_1,f_2$, such that for $i=1,2$, it holds
\begin{equation*}
\begin{cases}
\partial_t v_i(x,t)+\theta_i^*(\partial_x v_i(x,t))=0 & (x,t) \in
 \RR \times(0,\infty)\\v_i(x,0)=f_i(x) & x \in \RR,
\end{cases}
\end{equation*}
 $v=v_1+v_2$ and $f=f_1+f_2$.
\end{The}

We begin with a lemma.
\begin{Lem} \label{mapping}
Assume that $\theta$ is strictly convex, even and of class $C^1$. Let $f$ be a convex function bounded from below and of class $C^1$. For $t>0$, denote $I:=\{x\in \RR|f'(x)\in \theta'(\RR)\}$. Define the mapping $U_t$ as 
$$U_t: x\in I \mapsto x+t\theta'^{-1}\circ f'(x)\in \RR.$$
Then $T_t:=U_t^{-1}$ satisfies \eqref{star}.

\end{Lem}

\begin{proof}
Since $\theta$ is even, strictly convex and of class $C^1$, $\theta'(\RR)$ contains a neighborhood of $0$, combining with the fact that $f$ is bounded from below, convex and $C^1$, we deduce that $I$ is not an empty set. 
Since $\theta$ is strictly convex, $\theta':\RR\rightarrow \theta'(\RR)$ is a bijection and it is strictly increasing. Thus $\theta'^{-1}:\theta'(\RR)\rightarrow \RR$ is well defined as well as $U_t$.
It is easy to see that $U_t$ is strictly increasing and continuous.  Now we will show that the image set $U_t(I)=\RR$ and $U_t$ is in fact a bijection from $I$ to $\RR$.

Assume first that $\inf I=-\infty$, then it is easy to see that $\lim_{x\rightarrow -\infty} U_t(x)=- \infty$. Now assume that $\inf I > -\infty$. Then for $x\searrow\inf I$, $f'(x)\searrow \theta(-\infty)$, thus  $\theta'^{-1}f'(x)\rightarrow -\infty$. The same argument holds for $\sup I$. Therefore, $U_t:I\rightarrow \RR$ is bijective and strictly increasing. Thus $U_t^{-1}:\RR\rightarrow I$ is well defined and strictly increasing.

It remains to show that $U_t^{-1}(x)$ is the point achieving the infimum for $Q_t^\theta f(x)$.
For all $x\in \RR$, let $y=U_t^{-1}(x)$, using the fact that $\theta$ is even, it holds
$$\theta'\left(\frac{y-x}{t}\right)=\theta'\left(\frac{y-U_t(y)}{t}\right)=-\theta'\left(\theta'^{-1}\circ f'(y)\right)=-f'(y).$$
According to \eqref{carTt}, it holds $y=T_t(x)$.
\end{proof}

Now we are in position to prove proposition \ref{thm:Q+}.
\begin{proof}[Proof of proposition \ref{thm:Q+}]
We define $f_1$ as following: for all $y\in I$,
$$f_1(y):=\int_0^y \alpha' \circ \theta'^{-1}\circ f'(u)  du.$$
and (in the case that $I\neq \RR$) $f_1$ is affine when $y > \sup I$, and $y< \inf I$ with $f_1'=f_1'(\sup I):=\lim_{x\rightarrow \sup I} f_1'(x)$ and $f_1'=f_1'(\inf I)$ respectively.

We observe that $f_1$ is convex: for all $x\in I$, $f_1'=\alpha' \circ \theta'^{-1}\circ f'$ is non-decreasing, and for $x\in \RR\setminus I$, $f_1'$ is constant which equals to $f_1'=f_1'(\sup I)$ and $f_1'=f_1'(\inf I)$ respectively. It follows that $f_1'$ is non-decreasing on $\RR$.

Given $t>0$ and $x\in \RR$, denote $y=T_t(x)\in I$ and it holds
\begin{align*}
f_1'(y)&=\alpha' \circ \theta'^{-1}\circ f'(y)\\
&=-\alpha' \circ \theta'^{-1}\circ \theta'\left(\frac{y-x}{t}\right)=-\alpha'\left(\frac{y-x}{t}\right).
\end{align*}
We deduce from equation \eqref{carTt} that
\begin{equation}\label{eqalpha}
Q_t^{\alpha}f_1(x)= f_1(T_t(x))+t\alpha\left(\frac{T_t(x)-x}{t}\right).
\end{equation}

Now let $f_2=f-f_1$, together with $\theta=\alpha+\beta$, we deduce
$$f_2'(x)=\begin{cases}  f'-f_1'(\inf I)  & \forall x\in(-\infty, \inf I]\\ \beta'\circ \theta'^{-1}\circ f'(x) & \forall x\in I \\
  f'-f_1'(\sup I)  & \forall x\in [\sup I, \infty) \end{cases},$$
which is non-decreasing. Thus $f_2$ is convex.

On the other hand, since $f_1'(y)=-\alpha'\left(\frac{y-x}{t}\right)$  and $f'(y)=-\theta'\left(\frac{y-x}{t}\right)$, it holds that
$$f_2'(y)=-\beta'\left(\frac{y-x}{t}\right).$$
According to equation \eqref{carTt} and we have
\begin{equation}\label{eqbeta}
Q_t^{\beta}f_2(x)= f_2(T_t(x))+t\beta\left(\frac{T_t(x)-x}{t}\right).
\end{equation}
It immediately yields
\begin{equation*}
Q_t^\theta f=Q_t^\alpha f_1+Q_t^\beta f_2.
\end{equation*}
\end{proof}
Now we are in position to proof theorem \ref{add}.
\begin{proof}[Proof of theorem \ref{add}]
We first prove the case that the convex cost functions $\alpha$,$\beta$ and $\theta$ are $C^1$:

 We assume at first that $\theta=\alpha+\beta$ is strictly convex. 

Observe that by the definition of $\tbar$, it is easy to see that
\begin{equation}\label{sensefacile}
\tbar_{\theta}(\nu|\mu)\geq \tbar_\alpha(\nu|\mu)+\tbar_\beta(\nu|\mu).
\end{equation}
Now we turn to prove the inverse inequality.

According to Proposition \ref{thm:Q+}, for all convex function $f$ bounded from below and of class $C^1$, there exist $f_1$ and $f_2$ such that it holds:
\begin{align*}
\int Q^{\theta} fd\mu-\int fd\nu
&=\int Q^\alpha f_1d\mu-\int f_1d\nu + \int Q^\beta f_2d\mu-\int f_2d\nu\\
&\leq \tbar_\alpha(\nu|\mu)+\tbar_\beta(\nu|\mu).
\end{align*}
We take the supremum over all convex function $f$ bounded from below and of class $C^1$ and by the duality formula \eqref{duality}, we get
$$\tbar_{\theta}(\nu|\mu)\leq \tbar_\alpha(\nu|\mu)+\tbar_\beta(\nu|\mu).$$
The conclusion follows with the inverse inequality \eqref{sensefacile}.

Now assume that $\theta$ is not strictly convex. Since $\mu,\nu\in \mathcal{P}_1(\RR)$, there exists a strictly convex cost function $\gamma$, such that $\int \gamma(x-c) \mu(dx)$ and $\int \gamma(x-c) \nu(dx)$ are finite for all $c\in \RR$. Therefore $\tbar_{\theta+\gamma}(\nu|\mu)$ is finite.
Since $\alpha+\gamma$, $(\alpha+\gamma)+\beta$ are strictly convex, it holds:
\begin{equation}\label{eqthetagamma}
\tbar_{\theta+\gamma}(\nu|\mu)= \tbar_\theta(\nu|\mu)+\tbar_\gamma(\nu|\mu).
\end{equation}
and
\begin{align}\label{eqalphabetagamma}
\tbar_{\theta+\gamma}(\nu|\mu)&= \tbar_{(\alpha+\gamma)+\beta}(\nu|\mu) =\tbar_{\alpha+\gamma}(\nu|\mu)+\tbar_{\beta}(\nu|\mu) \nonumber \\
&=\tbar_{\alpha}(\nu|\mu)+\tbar_{\beta}(\nu|\mu)+\tbar_{\gamma}(\nu|\mu).
\end{align}
Combining equation \eqref{eqthetagamma} and \eqref{eqalphabetagamma}, we deduce that the equality \eqref{eqbut} holds for the case $\alpha,\beta$ and $\theta$ are $C^1$.

Now for general convex cost $\alpha, \beta$ and $\theta$, for all $\epsilon>0$ there exist cost functions $\alpha_\epsilon,\beta_\epsilon$ and $\theta_\epsilon$ satisfying $\theta_\epsilon=\alpha_\epsilon+\beta_\epsilon$ of class $C^1$ such that $$\|\alpha_\epsilon-\alpha\|_{\infty} + \|\beta_\epsilon-\beta\|_\infty + \|\theta_\epsilon-\theta\|_{\infty}<\epsilon.$$
We deduce that 
\begin{align*}
&\tbar_{\alpha}(\nu|\mu)+\tbar_{\beta}(\nu|\mu)\leq \tbar_{\theta}(\nu|\mu)\\
&\leq \tbar_{\theta_\epsilon}(\nu|\mu)+\epsilon= \tbar_{\alpha_\epsilon}(\nu|\mu)+\tbar_{\beta_\epsilon}(\nu|\mu)+\epsilon\\
&\leq \tbar_{\alpha}(\nu|\mu)+\tbar_{\beta}(\nu|\mu)+2\epsilon.
\end{align*}
Let $\epsilon$ goes to $0$ and the theorem follows.

\end{proof}
%
\subsection{An alternative approach}

The proof of theorem \ref{The:main1} in \cite{GRSST15} is based on properties of convex ordering and Rado's theorem. In this section, we provide another way to understand theorem~\ref{The:main1} in aspect of \cite{GRSST15}.

Here we only recall some necessary definitions and properties of the convex ordering and majorization of vectors.  We refer the interested reader to \cite{Majorization}, \cite{Hirsch} and \cite{Rad52} for further results and bibliographic references related to properties of inequalities, and to \cite{GRSST15, GJ2018, BBP2019} for recent developments related to the optimal weak transfer plan.

\begin{Def}[Convex order] Given $\nu_1,\nu_2 \in \mathcal{P}_1(\RR)$,  we say that $\nu_2$ dominates $\nu_1$ in the convex order,
and write $\nu_1\preceq\nu_2$, if for all convex functions $f$ on $\RR$,
$\int_\RR f\,d\nu_1\leq \int_\RR f\,d\nu_2$.
\end{Def}

\begin{Def}[Majorization of vectors]
Let $a,b\in \RR^n$, one says that $a$ is majorized by $b$ if the sum of the largest $j$ components of $a$ is less than or equal to the corresponding sum of $b$, for every $j$, and if the total sum of the components of both vectors are equal.
\end{Def}

Assuming that the components of $a=(a_1,\dots,a_n)$ and $b=(b_1,\dots,b_n)$ are in non-decreasing order (\textit{i.e.}\ $a_1 \leq a_2 \leq \dots \leq a_n$ and $b_1 \leq b_2 \leq \dots \leq b_n$), $a$ is majorized by $b$, if
\[
a_n + a_{n-1} + \cdots + a_{n-j+1} \le b_n + b_{n-1} + \cdots + b_{n-j+1}, \qquad \mbox{for } j=1, \dots,n-1,
\]
and $\sum_{i=1}^n a_i = \sum_{i=1}^n b_i$.

The next proposition recalls the link between majorization of vectors and convex ordering.

\begin{Pro}\label{maj=conv} \cite{GRSST15}
Let $a,b \in \RR^n$ and set $\nu_1 = \frac{1}{n}\sum_{i=1}^n \delta_{a_i}$ and $\nu_2=\frac{1}{n}\sum_{i=1}^n \delta_{b_i}$. The following statements are equivalent.
\begin{enumerate}
\item[(i)] $a$ is majorized by $b$.
\item[(ii)] $\nu_1$ is dominated by $\nu_2$ for the convex order. In other words, for every convex function $f:\RR \to \RR$, it holds that $\sum_{i=1}^n f(a_i) \le \sum_{i=1}^n f(b_i)\,.$
\end{enumerate}
\end{Pro}

Thanks to the above proposition and with a slight abuse of notation, we will also
write $a \preceq b$ when $a$ is majorized by $b$.


\medskip
Now we are able to prove an alternative version of theorem \ref{The:main1}. We shall focus on measure $\mu_n$ of form $\frac{1}{n}\sum_{i=1}^n\delta_{x_i}$. For general measure, one can consider it as a limit of $\mu_n$ as $n$ goes to $\infty$, readers can refer to \cite{GRSST15} for rigorous justification. Here is a discrete version of theorem \ref{The:main1} is telling the following:
\begin{The}[Discrete version of  theorem \ref{The:main1}]
Let $\mu=\frac{1}{n}\sum_{i=1}^n \delta_{x_i}$ and $\nu=\frac{1}{n}\sum_{i=1}^n \delta_{y_i}$, where $x_i$ and  $y_i$ are in non-decreasing order. Assume that $\theta$ is a strictly convex cost function. The following statements are equivalent.
 \begin{enumerate}
\item[(i)] The function $i\mapsto x_i-y_i$ is non-decreasing.
\item[(ii)]$\tbar_\theta(\nu|\mu)=\mathcal{T}_\theta(\nu,\mu)$.
\end{enumerate}
\end{The}

\begin{proof}
Observe that the optimal transfer plan of $\mathcal{T}_\theta(\nu,\mu)$ sends $x_i$ to $y_i$ since $x_i$ and  $y_i$ are in non-decreasing order. As a consequence,
$$\mathcal{T}_\theta(\nu,\mu)=\frac{1}{n}\sum_{i=1}^n \theta(x_i-y_i).$$
On the other hand, denoting $x=(x_1,...,x_n), y=(y_1,...,y_n)\in \RR^n$, according to \cite{GRSST15}, the following holds:
$$\tbar_\theta(\nu|\mu)=\inf_{\nu_1\preceq \nu}\mathcal{T}_\theta(\nu_1,\mu)= \inf_{y'\preceq y} \frac{1}{n}\sum_{i=1}^n \theta(x_i-y'_i).$$
Thus, item $(ii)$ is equivalent to
\begin{equation}\label{eqalter}
\sum_{i=1}^n \theta(x_i-y_i)=\inf_{y'\preceq y} \sum_{i=1}^n \theta(x_i-y'_i).
\end{equation}
Now it is enough to prove that \eqref{eqalter} is equivalent to $(i)$.

$(i)\Rightarrow (ii)$: For any $y'\preceq y$, it holds
$$\sum_{i=k}^n y'_i\leq \sum_{i=k}^n y_i \qquad \forall 1\leq k\leq n.$$
It follows that
$$\sum_{i=k}^n x_i-y_i \leq \sum_{i=k}^n x_i-y'_i \qquad \forall 1\leq k\leq n.$$
Thus $x-y\preceq x-y'$ for all $y'\preceq y$, which leads to \eqref{eqalter}.

$(ii)\Rightarrow (i)$: Assume that $(ii)$ holds and $(i)$ does not. Let $i$ be the smallest integer such that $x_i-y_i>x_{i+1}-y_{i+1}$.
Thus $ y_{i+1}-y_i-(x_{i+1}-x_i)=2\epsilon >0$.

Now define $y'\in \RR^n$ with $y'_j=y_j$ for $j\neq i,i+1$ $y'_i= y_i+\epsilon, y'_{i+1}=y_{i+1}-\epsilon$. It is easy to see that $y'\preceq y$ and by strict convexity of $\theta$:
\begin{align*}
&\sum_{i=1}^n \theta(x_i-y_i)-\sum_{i=1}^n \theta(x_i-y'_i)\\
&= \theta(x_i-y_i)+\theta(x_{i+1}-y_{i+1})-\theta(x_i-y_i-\epsilon)-\theta(x_{i+1}-y_{i+1}+\epsilon)\\
&= \theta(x_i-y_i)+\theta(x_{i+1}-y_{i+1})-2\theta\left(\frac{(x_i-y_i)+(x_{i+1}-y_{i+1})}{2}\right)\\
&>0
\end{align*}

which is a contradiction to \eqref{eqalter}.
\end{proof}

\section{Weak transport in $n$ dimensions}
This section is devoted to extending theorem \ref{cor:Q+} and theorem \ref{The:main1} of the previous section in $n$ dimensions for $n\geq 2$. In this section, we will focus on the space $\RR^n$ and an open subset $\Omega \subset \RR^n$, which correspond to the support of an absolutely continuous measure w.r.t Lebesgue measure.  We assume that all cost functions are convex, only depending on Euclidean distance $\|.\|:=\|.\|_2$ and vanishing in $0$. We denote $\mathcal{L}$ for the Lebesgue measure. 
%

In one dimension, convexity of functions plays a central role in the proofs, but in $n$ dimensions, we will replace the convexity by another stronger property, we call it property $\mathcal{F}$, which is defined as following.

\begin{Def}\label{propertyF}
Let $\Omega\subset \RR^n$ be a convex open set satisfying $\mathcal{L}(\Omega)>0$. We say that $f:\Omega \mapsto \RR$ satisfies the Property  $\mathcal{F}$ on $\Omega$ if the following conditions hold:
\begin{enumerate}
\item $f$ is convex on $\Omega$.
\item For almost all $x\in \Omega$, there exists $\lambda$, such that $\mathrm{Hess} f(x)\nabla f(x)=\lambda \nabla f(x)$. 
\end{enumerate}
\end{Def}
\begin{Rem}
Theorem of Alexandrov \cite{A39} guarantees that the Hessian of a convex function $f$ is well defined on $\RR^n$ almost everywhere.
\end{Rem}

\begin{Rem}
For dimension $n=1$, the Property $\mathcal{F}$ is the convexity.
\end{Rem}

Now let $f:\Omega\rightarrow \RR$ be a continuous function. Denote 
$$\mathcal{H}_f:=\{x\in \Omega, \mathrm{Hess} f(x)\ exists\}.$$ 
Due to regularity issues, we introduce the following smoothness assumption {\bf(SA)}.
{\bf(SA)}: The interior of $\mathcal{H}_f$, denoted by $\interior{\mathcal{H}_f}$, is simply connected, and satisfies 
$$\mathcal{L}(\Omega\setminus \interior{\mathcal{H}_f})=0.$$ 

Now we are ready to state the main theorem of this section, which is an extension of theorem \ref{The:main1} in $n$ dimensions.

\begin{The}\label{Themain1dim_n}
	Let $\mu,\nu\in \mathcal{P}_1(\RR^n)$ be two probability measures absolutely continuous with respect to Lebesgue measure. Assume that the support of the measure $\nu$ is a convex set $\Omega$.  Denote $g$ the convex function such that $\nabla g$ is the transfer plan for Monge-Kantorovich problem of quadratic cost, and satisfying $\nabla g \# \nu = \mu $. If $g$ satisfies the smoothness assumption $\mathrm{(SA)}$, then the following assertions hold.
	\begin{itemize}
		\item[(i)] If there exists a two times differentiable cost function $\theta$ such that $x\theta''(u)-\theta'(u)\neq 0$ for all $u>0$ and  $$\tbar_\theta(\nu|\mu)=\mathcal{T}_\theta(\mu,\nu)=\int \theta (\|x-\nabla g(x)\|)d\nu<\infty.$$ Then the function $x\mapsto g(x)-\frac{1}{2}\langle x,x\rangle$ satisfies the Property $\mathcal{F}$ on $\Omega$, and $\nabla g$ is an optimal transfer plan of $\mathcal{T}_\alpha$ for all convex cost $\alpha$ (such that $\tbar_\theta(\nu|\mu)$ is finite). Since the transport mapping $\nabla g$ of the quadratic cost is unique, $\nabla g$ is the unique optimal transport plan in common for all convex cost $\alpha$.
		
		\item[(ii)] Inversely, if $x\mapsto g(x)-\frac{1}{2}\langle x,x\rangle$ is strictly convex and satisfies Property $\mathcal{F}$, then for all convex cost $\theta$ such that $\tbar_\theta(\nu|\mu)$ is finite, $\nabla g$ is an optimal transfer plan of $\mathcal{T}_\theta$ for all convex cost $\theta$ and it holds $$\tbar_\theta(\nu|\mu)=\mathcal{T}_\theta(\mu,\nu).$$
		
	\end{itemize}
\end{The}
\begin{Rem}
	The existence and uniqueness of the optimal transport mapping $\nabla g$ of the quadratic cost is guaranteed by the Benamou-Brenier theorem \cite{BB2000}. 
\end{Rem}

\begin{Rem}
	In one dimension, adapting the notation in theorem \ref{The:main1}, the mapping $\nabla g$ is correspond to $F_\mu \circ F_\nu^{-1}$. A simple computation shows that the convexity of $g(x)-\frac{1}{2}x^2$ is equivalent to the fact that $F_\mu^{-1}-F_\nu^{-1}$ is non-decreasing. Therefore the equality between the weak transport cost and the Monge-Kantorovich transport cost is equivalence to the convexity of $g(x)-\frac{1}{2}x^2$. The latter theorem shows that with the additional smoothness assumption $\mathrm{(SA)}$, this equivalent still holds. Heuristically, I do believe that it holds true without $\mathrm{(SA)}$, but the regularity issue is very delicate and there is still a lot to investigate on.  
\end{Rem}

In order to prove the theorem, we begin with studying Property $\mathcal{F}$.
\subsection{Characterization of Property  $\mathcal{F}$}
Let us begin with some properties of the set $\mathcal{F}$:

\begin{Lem}\label{lem1}
Let $\Omega\subset \RR^n$ be a convex open set. Let $f:\Omega \rightarrow \RR$ be a continuous function. 
Denote $u=\|\nabla f\|$.
Then for all $x\in \mathcal{H}_f$, the following conditions are equivalent:

$(i)$ there exists $\lambda \in \RR$ such that
$$\mathrm{Hess} f (x) \nabla f(x) =\lambda \nabla f(x). $$

$(ii)$ for all $1\leq i\leq j \leq n$, if $\|\nabla f\|(x)\neq 0$, it holds
$$\partial_i u\partial_j f=\partial_j u \partial_i f.$$
\end{Lem}

From this lemma, we can construct some examples of $f$ such that $f\in\mathcal{F}$, by simply checking $(ii)$.
\begin{Exe}
In $n$ dimensions, following functions satisfy Property  $\mathcal{F}$ on $\RR^n$ :
\begin{enumerate}
\item Linear forms.
\item Functions of form $x\in \RR^n\mapsto g(\|x\|)$ where $g:\RR\rightarrow \RR$ is a convex function with $g(0)=0$. 
\item  Functions of form $x\in \RR^n \mapsto a\|x\|^2+L(x)$ with $a>0$ and $L:\RR^n\rightarrow \RR$ being a linear form.
\end{enumerate}
\end{Exe}

\begin{proof}[Proof of lemma \ref{lem1}]
$(i)\Rightarrow (ii)$:

$u$ is well defined on $\mathcal{H}_f$. Suppose that $(i)$ holds. For all $x\in \mathcal{H}_f$ with $\nabla f(x)\neq 0$, it holds for all $1\leq i \leq j \leq n$,
\begin{equation}\label{eqvap2}
\partial_i u=\frac{1}{u}\sum_{k=1}^n\partial_k f \partial_{ki} f=\frac{\lambda}{u} \partial_i f.
\end{equation}
Thus
$$\partial_i u\partial_j f=\frac{\lambda}{u} \partial_i f\partial_j f= \partial_j u \partial_i f.$$
$(ii)\Rightarrow (i)$:
For all $x\in \mathcal{H}_f$ such that $\nabla f(x)\neq 0$,
there exists $j$ such that $\partial_j f\neq 0.$ Let $\lambda = \partial_ju/\partial_j f$, $(ii)$ implies that for all $1\leq i\leq n$,
\begin{equation}\label{eqvap}
\partial_i u=\lambda  \partial_i f.
\end{equation}
 Computing the differential of $u$, it holds
 $$\partial_i u=\frac{1}{u}\sum_{k=1}^n\partial_k f \partial_{ki} f.$$
Together with \eqref{eqvap}, it holds for all $1\leq i\leq n$
$$\sum_{k=1}^n \partial_{ki} f\partial_k f= \lambda u \partial_i f,$$
which means exactly $\mathrm{Hess} f\ \nabla f =\lambda u \nabla f$.
\end{proof}


\begin{Pro}\label{lemcondition}
Let $\Omega\subset \RR^n$ be a convex open set and $f:\Omega\rightarrow \RR$ be a strictly convex function satisfying the smoothness assumption $\mathrm{(SA)}$.

Then the following statements hold.

$(i)$ If $f$ satisfies Property $\mathcal{F}$ on $\Omega$, then for all non-decreasing, differentiable function $G:\RR^+\rightarrow \RR^+$,
there exists $\phi \in \mathcal{F}$ such that for $x\in \Omega$ almost everywhere, it holds
 \begin{equation}\label{eqdim2}
\nabla \phi (x) = G (\|\nabla f\|(x))\frac{\nabla f(x)}{\|\nabla f\|(x)}.
\end{equation}

$(ii)$ Inversely, if there exists a non-decreasing differentiable function $G:\RR^+\mapsto \RR^+$ which satisfies $uG'(u)-G(u)\neq 0$ for all $u> 0$, and there exists a convex function $\phi :\Omega \rightarrow \RR$  such that
equation \eqref{eqdim2} holds almost everywhere, then $f$ satisfies Property $\mathcal{F}$ on $\Omega$.
\end{Pro}

\begin{proof}
The proof of $(i)$ is divided into three steps: we will prove at first the existence of a such function $\phi$ satisfying \eqref{eqdim2}, secondly we will show that $\phi$
is convex, at the end we will prove that $\phi\in \mathcal{F}$.

{\bf Step 1.} \textit{Existence of function $\phi$ such that \eqref{eqdim2} holds.}

By strict convexity of $f$, there exists at most one $x^*$ such that $\nabla f(x^*)=0$.

Denote $\mathbf{F}$ the vector field  $G(\|\nabla f\|)\frac{\nabla f}{\|\nabla f\|}$ defined on the simply connected open set $\interior{\mathcal{H}_f}\setminus{x^*}$, then the $i^{th}$ component of $\mathbf{F}$ is :
\begin{equation}\label{eqF}
\mathbf{F}_i=G(\|\nabla f\|)\frac{\partial_i f}{\|\nabla f\|}.
\end{equation}
We begin with proving the existence of function $\phi$ defined on $\interior{\mathcal{H}_f}\setminus\{x^*\}$, according to a generalized version of Poincar\'e's Lemma \cite{M08} and the smoothness assumption $\mathrm{(SA)}$, we only need to show that on  $\interior{\mathcal{H}_f}$, it holds:
$$\text{rot}\  \mathbf{F} =0,$$
which is equivalent to
\begin{equation}\label{eqrot}
\partial_j \mathbf{F}_i=\partial_i \mathbf{F}_j
\end{equation}
for all $1\leq i,j\leq n$.

Now denote $u=\|\nabla f\|_2$, it holds
\begin{align*}
\partial_j \mathbf{F}_i &= \partial_j\left(\frac{G(u)}{u}\right)\partial_i f+\frac{G(u)}{u}\partial_{ij} f \\
&= \partial_j u \left(\frac{G'(u)u-G(u)}{u^2}\right) \partial_i f +\frac{G(u)}{u}\partial_{ij} f
\end{align*}
The same argument leads to
\begin{equation}\label{eqrot3}
\partial_i \mathbf{F}_j= \partial_i u \left(\frac{G'(u)u-G(u)}{u^2}\right) \partial_j f +\frac{G(u)}{u}\partial_{ij} f.
\end{equation}
Since $f\in \mathcal{F}$, applying lemma \ref{lem1}, we deduce that
\begin{equation}\label{equf}
\partial_j u\ \partial_i f=\partial_i u\ \partial_j f.
\end{equation}
Thus equation \eqref{eqrot} holds and we have the existence of $\phi$ on $\interior{\mathcal{H}_f}\setminus\{x^*\}$. 

According to the assumption $\mathrm{(SA)}$, $\Omega\setminus \left(\interior{\mathcal{H}_f}\setminus\{x^*\}\right) $ is a finite set, together with the fact that $\Omega$ is an open set, one can extend the continuous function $\phi$ on $\Omega$. 

{\bf Step 2.} \textit{$\phi$ is convex.}

We only need to check $\mathrm{Hess}\ \phi$ is positive almost everywhere on $\Omega$.

Fix $x\in \interior{\mathcal{H}_f}\setminus\{x^*\} $, assume that $\mathrm{Hess}\ f  \nabla f =\lambda \nabla f$.  It is enough to prove that $\mathrm{Hess}\ \phi =(\partial_{ij} \phi)_{ij}$ is positive.  According to \eqref{eqrot3} and \eqref{eqvap2},  we have
\begin{align*}
&\partial_{ij}\phi=\partial_ i \mathbf{F}_j \\
&= \partial_i u \left(\frac{G'(u)u-G(u)}{u^2}\right) \partial_j f +\frac{G(u)}{u}\partial_{ij} f \\
&=\frac{G(u)}{u}\left(\partial_{ij} f-\frac{\partial_i u \partial_j f}{u} \right) + \frac{G'(u)}{u^2} \partial_i u\ \partial_j f\\
&=\frac{G(u)}{u}\left(\partial_{ij} f-\lambda \frac{\partial_i f \partial_j f}{u^2} \right) + \lambda \frac{G'(u)}{u^3} \partial_i f\ \partial_j f.
\end{align*}
Observe that $G$ is positive, non-decreasing and $u>0$, it follows that $\frac{G(u)}{u}>0$ and $\frac{G'(u)}{u^3}\geq 0$. Moreover, the convexity of $f$ implies that $\lambda \geq 0$. Thus, it is enough to prove that the matrix $M_0=\left(\partial_{ij} f-\lambda \frac{\partial_i f \partial_j f}{u^2} \right)_{ij}$ and $M_1=(\partial_i f\ \partial_j f)_{ij}$ are positive.

We begin with the positivity of $M_1$. For any vector $w\in \RR^n$,
\begin{equation}\label{hahaha}
^t wM_1 w =\sum_{i,j} w_i \partial_if\partial_j f w_j=\left(\sum_i w_i\partial_i f\right)^2=\langle w,\nabla f\rangle^2 \geq 0.
\end{equation}

Now we turn to prove the positivity of $M_0$.  For any $w\in \RR^n$, write $w=y+a\nabla f$ with $a\in \RR$ and $y$ perpendicular to $\nabla f$.
Noticing that $M_0=\mathrm{Hess} f-\frac{\lambda}{u^2} M_1$, then
$$^t wM_0 w=  ^t(y+a\nabla f) \left(\mathrm{Hess} f-\frac{\lambda}{u^2} M_1\right) (y+a\nabla f).$$
Using $\mathrm{Hess} f\ \nabla f =\lambda \nabla f$ and $u^2=\|\nabla f\|^2$, it holds
\begin{align*}
& ^t(y+a\nabla f) \mathrm{Hess} f \ (y+a\nabla f) \\
&= \langle y, \mathrm{Hess} f\ y \rangle+ a^2 \langle\nabla f, \mathrm{Hess} f \nabla f\rangle +2a \langle y, \mathrm{Hess} f \nabla f\rangle\\
&= \langle y, \mathrm{Hess} f\ y \rangle+ a^2 \langle\nabla f, \lambda \nabla f\rangle +2a \langle y, \lambda \nabla f\rangle\\
&= \langle y, \mathrm{Hess} f\ y \rangle+ a^2 \lambda u^2.
\end{align*}
On the other hand, according to \eqref{hahaha},
\begin{align*}
^t(y+a\nabla f) \frac{\lambda}{u^2} M_1 (y+a\nabla f)=\frac{\lambda}{u^2}\langle y+a\nabla f, \nabla f\rangle^2= \frac{\lambda a^2}{u^2}\langle\nabla f,\nabla f\rangle^2 = \lambda a^2 u^2.
\end{align*}
Thus, together with the convexity of $f$, we deduce that
$$^t wM_0 w=  ^t(y+a\nabla f) \left(\mathrm{Hess} f-\frac{\lambda}{u^2} M_1\right) (y+a\nabla f)= \langle y, \mathrm{Hess} f\ y \rangle\geq 0.$$
Hence, $M_0, M_1$ are positive matrices, and $\mathrm{Hess}\ \phi=\frac{G(u)}{u}M_0+\lambda \frac{G'(u)}{u^3}M_1$ is positive.

{\bf Step 3.} $\phi \in \mathcal{F}$.

The convexity of $\phi$ is proved in step 2. It is enough to show that for all $x\in \Omega$ almost everywhere, there exists $\lambda \in \RR$ such
that $\mathrm{Hess}\phi(x) \nabla\phi(x)=\lambda \nabla \phi (x)$.
Adapting the notations in step 1, applying lemma \ref{lem1}, it is enough to show that on $\interior{\mathcal{H}_f}\setminus\{x^*\}$, it holds
\begin{equation}\label{eqFij}
\partial_i \|F\| F_j=\partial_j\|F\| F_i.
\end{equation}
We develop $\partial_i \|F\| F_j$ by  \eqref{eqF} and \eqref{eqrot3}:
\begin{align*}
F_j\partial_i \|F\|
&=\frac{G(u)}{u} \partial_j f \frac{1}{\|F\|}\sum_{k=1}^n \frac{G(u)}{u}\partial_k f\left[\partial_i u\left(\frac{G'(u)-G(u)}{u^2}\right)\partial_k f+\frac{G(u)}{u}\partial_i f\right]\\
&=A_{ij}+B_{ij},
\end{align*}
where
$$A_{ij}:= \frac{G(u)}{u} \partial_j f \frac{1}{\|F\|}\sum_{k=1}^n \frac{G(u)}{u}\partial_k f\left[\partial_i u\left(\frac{G'(u)-G(u)}{u^2}\right)\partial_k f\right]$$
and
$$B_{ij}:=\frac{G(u)}{u} \partial_j f \frac{1}{\|F\|}\sum_{k=1}^n \left(\frac{G(u)}{u}\right)^2\partial_k f \partial_i f.$$
It is easy to see that $B_{ij}=B_{ji}$. Using \eqref{equf}, we deduce that $A_{ij}=A_{ji}$. Therefore, equation \eqref{eqFij} holds and it follows that $\phi \in \mathcal{F}$. The proof of item $(i)$ is completed.

Now we turn to prove item $(ii)$. Adapting the notations in the step 1 of the proof of item $(i)$, the existence of $\phi$ guarantees that
$$\mathrm{rot}\  \mathbf{F}=0,$$
Developing the latter equation (see \eqref{eqrot3}), it holds
$$\partial_i u \left(\frac{G'(u)u-G(u)}{u^2}\right) \partial_j f +\frac{G(u)}{u}\partial_{ij} f=\partial_j u \left(\frac{G'(u)u-G(u)}{u^2}\right) \partial_i f +\frac{G(u)}{u}\partial_{ij} f.$$
By assumption of $G$, we deduce that for almost all $x\in \Omega$,
$$\partial_i u\partial_j f=\partial_j u\partial_i f.$$
The conclusion follows by applying lemma \ref{lem1}.
\end{proof}

Now we present extensions of proposition \ref{thm:Q+} and theorem \ref{cor:Q+} in the case of dimension $n$.
\begin{Pro}\label{adddimn}
Let $\alpha, \beta, \theta$ be differentiable strictly convex cost functions such that $\alpha + \beta=\theta$. We assume that $\theta$ is strictly convex. Then if $f$ is differentiable and satisfying Property $\mathcal{F}$ on $\Omega$, there exists $\phi, \psi\in \mathcal{F}$,  such that
\begin{equation*}
f=\phi+\psi
\end{equation*}
and for all $t>0$, it holds for $x\in \Omega$ almost surely,
\begin{equation*}
Q_t^\theta f(x)=Q_t^\alpha \phi(x)+Q_t^\beta \psi(x)
\end{equation*}
\end{Pro}

\begin{proof}
Since $f$ satisfying Property $\mathcal{F}$ and $\mathrm{(SA)}$  on $\Omega$,
according to Proposition \ref{lemcondition} for $G=\alpha'\circ{\theta'^{-1}}$, there exists a $\phi\in \mathcal{F}$ such that
\begin{equation}\label{ei}
 \nabla \phi(x)= \alpha' \circ \theta'^{-1}(\|\nabla f\|(x))\frac{\nabla f(x)}{\|\nabla f\|(x)}
\end{equation}
holds for $x\in \Omega$ almost surely.

Denote $T_t(x)$ the point such that
$$Q_t^\theta f(x)=f(T_t(x))+t\theta\left(\frac{\|x-T_t(x)\|}{t}\right).$$
It follows that  $x\in \Omega$ almost surely,
$$\nabla f(T_t(x))=\theta'\left(\frac{\|x-T_t(x)\|}{t}\right)\frac{x-T_t(x)}{\|x-T_t(x)\|}.$$
Combining equation \eqref{ei}, we get
\begin{align*}
\nabla \phi(T_t(x))&= \alpha' \circ \theta'^{-1}(\|\nabla f\|)\frac{\nabla f}{\|\nabla f\|}(T_t(x))\\
&=\alpha'\left(\frac{\|x-T_t(x)\|}{t}\right)\frac{x-T_t(x)}{\|x-T_t(x)\|},
\end{align*}
which implies that for all $x\in \RR^n$,
\begin{equation}\label{q alpha}
Q_t^\alpha\phi(x)=\phi(T_t(x))+t\alpha\left(\frac{\|x-T_t(x)\|}{t}\right).
\end{equation}
Now letting 
$$\psi=f-\phi,$$
it holds
$$\nabla \psi(T_t(x))=\beta'\left(\frac{\|x-T_t(x)\|}{t}\right)\frac{x-T_t(x)}{\|x-T_t(x)\|}.$$
It follows that
\begin{equation}\label{q beta}
Q_t^\alpha\psi(x)=\psi(T_t(x))+t\beta\left(\frac{\|x-T_t(x)\|}{t}\right).
\end{equation}
Summing \eqref{q alpha} and \eqref{q beta} leads to the conclusion.
\end{proof}

\subsection{Proof of Theorem \ref{Themain1dim_n}}    

Now we are ready to prove theorem \ref{Themain1dim_n}.
\begin{proof} 
We first prove $(ii)$:
Since $g$ satisfies the smoothness assumption $\mathrm{(SA)}$ and 
$$x\mapsto g(x)- \frac{1}{2}\langle x,x\rangle$$ is strictly convex, according to Proposition \ref{lemcondition} by taking $f:=x\mapsto g(x) -\frac{1}{2}\langle x,x \rangle $ and $G= \theta'$,
there exists a convex function $\phi$ such that it holds for $x\in \Omega$ almost everywhere,
$$\nabla \phi(x) =\theta'(\|x-\nabla g(x)\|)\frac{\nabla g(x)-x}{\|x-\nabla g(x)\|}.$$
Since $\Omega$ is the support of $\mu$,  it follows that $\nu$ almost surely
$$Q^\theta \phi(\nabla g (x))=\phi(x)+\theta(\|x-\nabla g (x)\|).$$
Hence,
\begin{align*}
\mathcal{T}_\theta(\mu,\nu)&\leq\int \theta(\|x-\nabla g(x)\|)d\nu\\
&=\int Q^\theta \phi(\nabla g(x))- \phi(x) d\nu\\
&=\int Q^\theta \phi d\mu -\int \phi d\nu \leq \tbar_\theta (\nu|\mu).
\end{align*}
Together with the fact that $\tbar_\theta (\nu|\mu)\leq \mathcal{T}_\theta(\mu,\nu)$, the equality holds.

Now we prove that the transfer plan does not depend on the cost function.

Since $\theta$ is two times differentiable and convex, applying proposition \ref{lemcondition} to equation \eqref{eqhh} for $G=\theta'$, together with the fact that $x\rightarrow g(x)-\frac{1}{2}\langle x, x\rangle\in \mathcal{F}$, we deduce that $\phi\in \mathcal{F}$.
Now given a convex cost function $\alpha$, applying  Proposition \ref{lemcondition} for $G=\alpha'\circ\theta'^{-1}$, we deduce that there exists $\psi\in \mathcal{F}$ such that for $\nu$ almost surely,
$$\nabla \psi(x)=\alpha'\circ \theta'^{-1}\frac{\nabla \phi(x)}{\|\nabla \phi(x)\|}.$$
It follows that $\nu$ almost surely,
$$Q^\alpha \psi(\nabla g (x))=\psi(x)+\alpha(\|x-\nabla g (x)\|).$$
The conclusion follows by writing the definition of optimal transport and its Kantorovich's duality:
$$\mathcal{T}_\alpha(\mu,\nu)=\inf_\pi\left\{\int \alpha (\|x-y\|)d\pi\right\} \leq \int \alpha(\|x-\nabla g(x)\|)d\nu,$$
and
\begin{align*}
\mathcal{T}_\alpha(\mu,\nu)&=\sup_f\left\{\int Q^\alpha f d\mu -\int f d\nu \right\}\\
&\geq \int Q^\alpha \psi d\mu -\int \psi d\nu\\
&=\int \alpha(\|x-\nabla g(x)\|)d\nu.
\end{align*}
By a similar argument with the suprement taken over all $f$ convex, Lipschitz and lower semi-continuous, it holds
$$\tbar_\alpha(\nu|\mu)=\int \alpha(\|x-\nabla g(x)\|)d\nu.$$

We now turn to prove $(i)$:
For all convex function $\phi$, it holds for all $x\in \Omega$,
\begin{equation}\label{eqnn}
Q^\theta \phi(\nabla g(x))\leq \phi(x)+\theta(\|x-\nabla g (x)\|).
\end{equation}
Now let $\phi$ be the convex function such that
$\tbar_\theta(\nu|\mu)= \int Q^\theta \phi d\mu -\int \phi d\nu$.
Together with \eqref{eqnn}, it follows
\begin{align*}
\tbar_\theta(\nu|\mu)&= \int Q^\theta \phi d\mu -\int \phi d\nu\\
&\leq \int \theta(\|x-\nabla g (x)\|)d\nu =\mathcal{T}_\theta(\mu,\nu).
\end{align*}
The assumption of $(i)$ implies that the inequality in the latter formula is in fact equality. Thus, for $x\in \Omega$, $\nu$ almost surely,
equation \eqref{eqnn} holds. We deduce that
\begin{equation}\label{eqhh}
\nabla \phi(x) =\theta'(\|x-\nabla g(x)\|)\frac{\nabla g(x)-x}{\|x-\nabla g(x)\|}.
\end{equation}
According to Proposition \ref{lemcondition}, the conclusion follows.

\end{proof}
\section{Applications}
\subsection{A simple example}
In the space $\RR^n$, let $\mu$ be the uniform probability measure on the unit ball $B(0,1)$ and $\nu$ be the uniform probability measure on the ball $B(a,1/4)$.   Observe that $T:=x\mapsto 4x-a$ satisfies $T\#\nu=\mu$ and define $g(x):=2\|x\|^2-ax$ for $x\in B(a,1/4)$. It is easy to check that $\nabla g=T$ and it is the optimal transfer plan $\mathcal{T}(\mu,\nu)$ with quadratic cost. Now observe that $g(x)-\frac{1}{2}\|x\|^2$ satisfies Property $\mathcal{F}$ on $B(a,1/4)$.  According to item $(ii)$ of  theorem~\ref{Themain1dim_n}, we can deduce that $T$ is an optimal transfer plan for both weak and classical transport problem with convex costs depending on the distance.    

\subsection{Links with the infimum convolution inequality}
The so-called infimum operator inequalities were first introduced by Maurey in \cite{Mau91}. They are closely related to Transport-cost inequalities.

Let us stay in the space $\RR^n$ and adapt the settings before. We say that a probability measure $\mu$ satisfies the inf-convolution inequality $\mathrm{IC}(\theta)$ with the cost $\theta$ if the following holds for all measurable functions bounded from below $f:\RR^n\mapsto \RR$:
$$\int e^{Q^\theta f} d\mu \int e^{-f}d\mu \leq 1.$$
This inequality was proved to be equivalent to the transport cost inequality (see \cite{BG99a}):
$$\mathcal{T}_\theta (\nu,\mu)\leq H(\nu|\mu),$$
where $H(\nu|\mu)$ is the related entropy of $\nu$ with respect to $\mu$.

 Now consider the inf-convolution inequality restricted to the class $\mathcal{F}$ (denoted by $\mathrm{rIC}(\theta)$):
$$\int e^{Q^\theta f} d\mu \int e^{-f}d\mu \leq 1 \qquad \forall f\in \mathcal{F}.$$
According to proposition \ref{adddimn}, let  $\alpha$, $\beta$ be convex costs such that $\alpha+\beta=\theta$, assume that $\mathrm{rIC}(\alpha)$ and $\mathrm{rIC}(\beta)$ hold, then $\mathrm{rIC}(\frac{1}{2}\theta)$ holds. The proof is simply apply Cauchy Schwartz inequality and proposition \ref{adddimn}, details are left to the readers.

We remark that in one dimension, Property $\mathcal{F}$ is convexity. Feldheim and al. in \cite{FMNW15} and Gozlan and al. in \cite{GRSST15} proved independently that for a quadratic linear cost $\alpha$, the inequality $\mathrm{rIC}(\alpha)$ is equivalent to the convex Poincar\'e inequality. For general convex cost $\theta$, $\mathrm{rIC}(\theta)$ is equivalent to the weak transport inequality $\tbar_\theta \leq H$. The fact that $\mathrm{rIC}(\alpha)$ and $\mathrm{rIC}(\beta)$ imply $\mathrm{rIC}(\theta)$ is simply $\tbar_\theta=\tbar_\alpha+\tbar_\beta$ in one dimension.

\section*{Acknowledgement}
I warmly thank my PhD advisers Natha\"el Gozlan and Cyril Roberto, and my friend Andrew Mittleider for helpful advice and remarks.

I sincerely thank anonymous referees for precious advice which helped to largely improve the article. 

\bibliographystyle{plain}

\end{document}